\documentclass{amsart}
\usepackage{amssymb, amsmath, tikz-cd, amscd}
\usepackage{stmaryrd}
\usepackage{xcolor}

\title{A remark on post-critically finite compositions of polynomials}
\author{Benjamin Fraser and Patrick Ingram}
\address{York University, Toronto, Canada}

\renewcommand{\epsilon}{\varepsilon}
\renewcommand{\phi}{\varphi}

\newcommand{\ZZ}{\mathbb{Z}}
\newcommand{\CC}{\mathbb{C}}

\newcommand{\QQ}{\mathbb{Q}}

\newtheorem{lemma}{Lemma}

\newtheorem{theorem}[lemma]{Theorem}

\theoremstyle{definition}

\begin{document}

\begin{abstract}
The second author proved that the set of post-critically finite polynomials of given degree is a set of bounded height, up to change of variables. Motivated by an observation about unicritical polynomials, we complement this by proving that the set of monic polynomials $g(z)\in \overline{\QQ}[z]$ of given degree with the property that there exists a $d\geq 2$ such that $g(z^d)$ is post-critically finite,  is also a set of bounded height. Moreover, we establish a lower bound on the critical height of $g(z^d)$.
\end{abstract}
\maketitle

A polynomial $f(z)\in \CC[z]$ is \emph{post-critically finite (PCF)} if and only if the orbit of every critical point is finite, and in this case one can show that $f$ is defined over $\overline{\QQ}$, after a change of variables. The first author showed that the set of PCF polynomials of given degree is a set of bounded height~\cite{pcf}, depending strongly on the degree. So for example, for each $d\geq 2$ and each number field $K$ there exist only finitely many $c\in K$ such that $z^d+c$ is PCF. In this very special case, though, one can do much better, and with elementary tools. If $z^d+c$ is PCF, then it is easy to show that $h(c)\leq \frac{\log 2}{d-1}$, and so in particular there are only finitely many $c\in K$ such that \emph{there exists a $d\geq 2$} with $z^d+c$ PCF. This note generalizes this phenomenon from unicritical polynomials, viewed as post-compositions of a power map by a linear map, to arbitrary polynomial post-compositions of a power map, including a lower bound on the critical height.
\begin{theorem}\label{th:main}
	For monic polynomials $g(z)\in\overline{\QQ}[z]$ and $d\geq 2$, we have
\[\hat{h}_{\mathrm{crit}}(g(z^d))\geq \frac{1}{d\deg(g)}\Big(h(g)+O(1)\Big),\]
where the implied constant depends just on $\deg(g)$. In particular, the set of monic $g(z)\in\overline{\QQ}[z]$ of given degree such that there exists a $d\geq 2$ with $g(z^d)$ PCF is a set of bounded height.
\end{theorem}

The main result of~\cite{pcf} is a similar lower bound on the critical height for polynomials of fixed degree (i.e., the $d=1$ case of Theorem~\ref{th:main}), but the height bound there is only up to change of coordinates, since $f(z+a)-a$ will be PCF whenever $f$ is. In contrast, a monic polynomial of the form $g(z^d)$ is only conjugate to finitely many other polynomials of the same form, when $d\geq 2$ (since the change of variables would need to move $z=0$ to another critical point).

It is well known that there are no algebraic families of PCF polynomials in characteristic zero, and the tools in~\cite{pcf} allow one to extend this result to characteristic $p$, but only for polynomials of degree less than $p$ (see also work of Levy~\cite{levy}). Indeed, the family $z^{p^m}+tz$ is PCF in characteristic $p$, for any $m$ as long as $t\neq 0$, so one cannot hope to completely remove the condition on the degree. One can, however, obtain results in the flavour of Theorem~\ref{th:main}.

\begin{theorem}\label{th:rigid}
	Let $k$ be an algebraically closed field of characteristic $p>0$. Then there are no non-constant PCF algebraic families of the form $g(z^d)$ with $d\geq 2$, and $\deg(g)<p$.
\end{theorem}

One must specify what one means by a PCF algebraic family, since it is of course the case that any rational function defined over $\overline{\mathbb{F}_p}$ is PCF. By a PCF algebraic family we mean  a family parametrized by an irreducible quasi-projective variety, whose generic fibre is PCF. It is easy to check that this is equivalent to saying that the postcritical sets of the specializations are of uniformly bounded size.	
%

\section{Local considerations}

In this section we fix $m\geq 2$, and work over a field $K$ of characteristic $0$ or $p>m$, equipped with a set $M$ of absolute values $|\cdot|_v$, extended in some way to the algebraic closure of $K$. For any tuple $x_1, ..., x_m$ we set $\mathbf{x}=x_1, ..., x_m$ and 
\[\|\mathbf{x}\|=\|x_1, ..., x_n\|_v=\max\{|x_1|_v, ..., |x_n|_v\},\]
and we concatenate tuples with a comma.

We also consider monic polynomials $g(z)$ of degree $m$, and write
\[g(z)=(z-a_1)\cdots(z-a_m)\]
and
\[g'(z)=m(z-c_1)\cdots(z-c_{m-1})\]
over $\overline{K}$.

\begin{lemma}
	We have
	\begin{equation}\label{eq:ac}\Big|\log\|\mathbf{a}\|_v-\log\|\mathbf{c}, g(0)^{1/m}\|_v\Big|\leq C_{1, v},\end{equation}
	where $C_{1, v}=0$ when $v$ is non-archimedean, or some fixed non-negative value depending just on $m$ otherwise.
\end{lemma}

\begin{proof}
If $\sigma_i$ is the basic symmetric polynomial of degree $i$ in the relevant number of variables, then
\[g(z)=z^m-\sigma_1(\mathbf{a})z^{m-1}+...\pm \sigma_m(\mathbf{a}),\]
and so
\[g'(z)=mz^{m-1}-(m-1)\sigma_1(\mathbf{a})z^{m-2}+\cdots \mp \sigma_{m-1}(\mathbf{a}).\]
In the non-archimedean case, $g'(c_i)=0$ implies
\[|mc_i^{m-1}|_v\leq |(m-j)\sigma_j(\mathbf{a})c_i^{m-j-1}|_v\leq |c_i|_v^{m-j-1}\|\mathbf{a}\|_v^j,\]
for some $j$,
whence $|c_i|_v\leq \|\mathbf{a}\|_v|m^{-1}|_v$. Since
\[|g(0)|_v=|a_1\cdots a_m|_v\leq \|\mathbf{a}\|_v^m,\]
we have one direction of our bound, and the archimedean case is similar, using
\[|\sigma_j(\mathbf{a})|_v\leq \binom{m}{j}\|\mathbf{a}\|_v^j.\]

In the other direction, we may use a similar argument applying $g(a_i)=0$ to the expression
\[g(z)=z^m-\frac{1}{m-1}\sigma_1(\mathbf{c})z^{m-1}+\cdots \mp \sigma_{m-1}(\mathbf{c})z+g(0).\]
\end{proof}

Now, for any polynomial $f(z)\in K[z]$ of degree at least 2, set
\[G_{f, v}(z)=\lim_{k\to\infty}\frac{\log^+|f^k(z)|_v}{\deg(f)^k},\]
noting that $G_{f, v}(f(z))=\deg(f)G_{f, v}(z)$,
and put
\[\lambda_{\mathrm{crit}, v}(f)=\max_{f'(c)=0}\{G_{f, v}(c)\}.\]

\begin{lemma}
Set $f(z)=g(z^d)$, and suppose that 
\begin{equation}\label{eq:basin}d\log|z|_v>\log^+\|\mathbf{a}\|_v+C_{2, v},\end{equation} where $C_{2, v}=\frac{m}{m-1}\log 2$ for $v$ archimedean, and $C_{2, v}=0$ otherwise. Then
\[G_{f, v}(z)=\log|z|_v+\epsilon_v(z),\]
where $\epsilon_v(z)=0$ for non-archimedean $v$, and \[-\frac{m}{dm-1}\log 2\leq \epsilon_v(z)\leq \frac{m}{dm-1}\log(3/2)\] for $v$ archimedean.	
\end{lemma}

\begin{proof}
First consider $v$ non-archimedean. Then $\log|z|_v^d>\log^+\|\mathbf{a}\|_v$ implies $|z^d-a_i|_v=|z|_v^d$ for all $i$, whence
\[\log|f(z)|_v=\deg(f)\log|z|_v\geq\log|z|_v^d>\log^+\|\mathbf{a}\|_v.\] Iterating gives the desired result.

For an archimedean $v$, \[\log|z|_v^d>\log^+\|\mathbf{a}\|_v+\frac{m}{m-1}\log 2>\log^+\|\mathbf{a}\|_v+\log 2\] gives $|z^d-a_i|_v>|z|_v^d/2$, whence
\begin{align*}
\log |f(z)|_v&\geq \deg(f)\log|z|_v-m\log 2\\
&\geq d\log|z|_v+(m-1)d\log|z|_v-m\log 2\\
&\geq d\log|z|_v+(m-1)\log^+\|\mathbf{a}\|_v\\
&\geq d\log|z|_v\geq\log\|\mathbf{a}\|_v+C_{2, v},
\end{align*}
and so iterating gives
\[G_{f, v}(z)\geq \log|z|_v-\frac{m}{dm-1}\log 2.\]
The other direction is obtained using the conclusion $|z^d-a_i|_v\leq \frac{3}{2}|z|_v^d$ from~\eqref{eq:basin}.
\end{proof}

We now prove the key lemma.

\begin{lemma}\label{lem:mainbound}
We have
\[\lambda_{\mathrm{crit}, v}(g(z^d))\geq \frac{1}{d}\Big(\log^+\|\mathbf{a}\|_v-C_{3, v}\Big),\]
where $C_{3, v}$ depends only on $m$ and the absolute value, and can be taken to be 0 if $v$ is non-archimedean, and not $p$-adic for any $p\leq m$.	
\end{lemma}

\begin{proof}
Note first, as a general comment, that $\lambda_{\mathrm{crit}, v}(g(z^d))\geq 0$, and so the desired result holds trivially if ever $\log^+\|\mathbf{a}\|_v\leq C_{3, v}$. We will assume that \begin{equation}\label{eq:somestuff}\log^+\|\mathbf{a}\|_v=\log\|\mathbf{a}\|_v> \frac{m}{m-1}C_{1, v}+\frac{1}{m}C_{2, v}\geq 0,\end{equation}
since requiring $C_{3, v}\geq \frac{m}{m-1}C_{1, v}+\frac{1}{m}C_{2, v}$ will mean that our claim is trivial otherwise.
	
 Write $f(z)=g(z^d)$, and note by the chain rule that the branch points of $f$ (the images of the critical points) are exactly $g(0)$ and the branch points of $g$. 
   Let $\beta$ be any branch point of $f$ and  suppose  for now that that \[\log|\beta|_v>\log\|\mathbf{a}\|_v+C_{2, v},\] so that $z=\beta$ certainly satisfies the hypothesis~\eqref{eq:basin}. Then we have
 \begin{equation}\label{eq:gzero}
 	\lambda_{\mathrm{crit}, v}(f)\geq \frac{1}{d}G_{f, v}(\beta)\geq \frac{1}{d}\Big(\log|\beta|_v+\epsilon_v(z)\Big)\geq \frac{1}{d}\Big(\log\|\mathbf{a}\|_v+C_{2, v}+\epsilon_v(z)\Big),
 \end{equation}
which implies the claimed bound, as long as we eventually take $C_{3, v}>\frac{m}{2m-1}\log 2$.

  So we are done unless \begin{equation}\label{eq:betabound}\log|\beta|_v\leq \log\|\mathbf{a}\|_v+C_{2, v}\end{equation} for every branch point $\beta$ of $f$. In particular, since $g(0)$ is a branch point, and \[m\log\|\mathbf{a}\|_v-mC_{1, v}>\log\|\mathbf{a}\|_v+C_{2, v},\] from~\eqref{eq:somestuff}, we have
  $\log|g(0)|< m\log\|\mathbf{a}\|-mC_{1,v}$, whence from~\eqref{eq:ac}
 \[\Big|\log\|\mathbf{a}\|-\log\|\mathbf{c}\|\Big|\leq C_{1, v}.\]

%
%
Now write $\psi(z)=\frac{1}{m}(g(z)-g(0))$. By the main result of~\cite{pcf}, there is a constant $C_{4, v}$, depending just on the absolute value (and to be taken as 0 if $v$ is non-archimedean and not $p$-adic for $p\leq m$; see e.g., \cite{pern}, and the comment therein that arguments pursued over $\QQ$ in~\cite{pcf} really work over $\ZZ[1/2, 1/3, ..., 1/m]$), such that
 \[\log|\psi(c_j)|\geq m\log\|\mathbf{c}\|-C_{4, v}\]
 for some $j$, noting that the critical points of $\psi$ are also the $c_i$. Thus we have
 \begin{align*}
m\log\|\mathbf{a}\|-mC_{1, v}-C_{4, v}-\log|m|&\leq m\log\|\mathbf{c}\|-C_{4, v}-\log|m|\\
&\leq \log|g(c_j)-g(0)|\\
&\leq \log\max\{|g(c_j)|, |g(0)|\}\\
&\leq \log\|\mathbf{a}\|+C_{2, v},
 \end{align*}
by~\eqref{eq:betabound}, which gives \begin{equation}\label{eq:4consts}
\log\|\mathbf{a}\|\leq \frac{1}{m-1}\Big(mC_{1, v}+C_{2, v}+C_{4, v}+\log|m|\Big).	
 \end{equation}
 Again our bound is trivial as long as we take $C_{3, v}$ to be at least as large as the right-hand-side of~\eqref{eq:4consts}.
\end{proof}

\section{Proofs of Theorems~\ref{th:main} and Theorem~\ref{th:rigid}}

It is now relatively easy to prove Theorem~\ref{th:main} from Lemma~\ref{lem:mainbound}. 
\begin{proof}[Proof of Theorem~\ref{th:main}]
	Recall that for a tuple $\mathbf{a}\in K^m$, we have
\[h(\mathbf{a})=\sum_{v\in M_K} \frac{[K_v:\QQ_v]}{[K:\QQ]}\log^+\|\mathbf{a}\|_v.\]
We define, as is now standard for polynomials,
\[\hat{h}_{\mathrm{crit}}(f)=\sum_{v\in M_K}\frac{[K_v:\QQ_v]}{[K:\QQ]}\lambda_{\mathrm{crit}, v}(f).\]
Summing Lemma~\ref{lem:mainbound} over all places, we have
\[\hat{h}_{\mathrm{crit}}(g(z^d))\geq \frac{1}{d}\Big(h(\mathbf{a})-C\Big),\]
where $C=\sum_{v\in M_K}\frac{[K_v:\QQ_v]}{[K:\QQ]}C_{3, v}$. Using the estimate
\[h(g)\leq \deg(g)(h(\mathbf{a})+\log 2)\]
(obtained by bounding the coefficients of $g$ as symmetric polynomials in the roots) completes the proof of Theorem~\ref{th:main}.
\end{proof}

\begin{proof}[Proof of Theorem~\ref{th:rigid}]
Recall (e.g., \cite[Lemma~10]{pern}) that if $V$ is any quasi-projective variety over an algebraically closed field $k$, then the function field $k(V)$ admits a set of absolute values satisfying the product formula, and such that the points of height zero are precisely the points with constant coordinates. Now, the error terms in Lemma~\ref{lem:mainbound} vanish for each of these absolute values, since they are neither archimedean nor $p$-adic, and so we have
\[\lambda_{\mathrm{crit}, v}(g(z^d))\geq \frac{1}{d\deg(g)}\log^+\|\mathbf{a}\|_v\]
for all places. So if $g(z^d)$ is PCF, we must have $\|\mathbf{a}\|_v\leq 1$ for all $v$, and so the $a_i$ are constant.
\end{proof}

We end with a remark, namely: 
Let $V$ be a quasi-projective variety over $k=\overline{\mathbb{F}_p}$, and suppose that $f(z)\in k[V, z]$ has degree $d\geq 2$ in $z$. Then $f$ is PCF if and only if there is uniform bound on the size of the postcritical set of the specialization $f_t$, for $t\in V(k)$. 

In one direction this is clear, since if the postcritical set on the generic fibre contains at most $N$ elements, the same is true of every specialization. On the other hand, for every $N$, there is a Zariski closed subset of $V$ corresponding to specializations in which the postcritical set contains at most $N$ elements. If the specializations have uniformly bounded postcritical set sizes, then since $\overline{\mathbb{F}_p}$ points are Zariski dense in $V$, it must be that some irreducible component of this closed subset is all of $V$. But, marking the critical points and describing their orbits partitions these components, so all orbits must have the same combinatorics on $V$.

\end{document}